\documentclass[11pt]{amsart}
\usepackage{amssymb}
\usepackage{amsmath}
\usepackage{tikz}
\usepackage{hyperref}

\hypersetup{
	pdfstartview={XYZ null null 1.00}, 
	pdfpagemode=UseNone, 
	colorlinks,
	breaklinks, 
	linkcolor=blue,
	urlcolor=blue, 
	anchorcolor=blue,
	citecolor=blue
}

\newtheorem{theorem}{Theorem}[section]
\newtheorem{question}{Question}
\newtheorem{lemma}[theorem]{Lemma}

\newtheorem{cor}[theorem]{Corollary}

\newtheorem{claim}[theorem]{Claim}

\theoremstyle{definition}
\newtheorem{definition}[theorem]{Definition}

\DeclareMathOperator{\supp}{supp}

\newcommand{\res}{\upharpoonright}

\title{Measurable Regular Subgraphs}
\author{Matt Bowen, Clinton T.~Conley, Felix Weilacher}
\thanks{The first author is supported by Ben Green's Simons Investigator Grant number 376201.  The second author is supported by NSF grant DMS-2154160.
The third author is supported by the NSF under award number DMS-2402064.}

\begin{document}

\begin{abstract}

    We show that every $d$-regular bipartite Borel graph admits a Baire measurable $k$-regular spanning subgraph if and only if $d$ is odd or $k$ is even.  {This gives the first example of a locally checkable coloring problem which is known to have a Baire measurable solution on Borel graphs but not a computable solution on highly computable graphs.}  We also prove the analogous result in the measure setting for hyperfinite graphs.
\end{abstract}

\maketitle

\section{Introduction}


A $\mathbf{k}$\textbf{-factor} is a $k$-regular spanning subgraph of a given graph $G.$ For example, a 1-factor is what is usually called a perfect matching. Characterising the existence of $k$-factors is one of the foundational problems in graph theory, originating with Peterson's $2$-factor theorem \cite{petersen1891theorie} and culminating in Tutte's classification \cite{tutte1954short}.  In the special case of bipartite graphs, K\H{o}nig's edge coloring theorem is equivalent to the existence of $k$-factors in $d$-regular bipartite graphs for all $k\leq d.$

Here, we will be interested in measurable analogous of the factorization problem.  A \textbf{Borel graph} is a graph whose vertex set is a standard Borel space $(X,\mathcal{B})$ and whose edge set is a Borel subset of $X^2.$  The existence of Borel $k$-factors is more subtle than in the classical setting.  Namely, Laczkovich \cite{lacz3} constructed a $2$-regular Borel bipartite graph with no Borel $1$-factor, and Conley and Kechris \cite{conley2013measurable} extended this construction to find $d$-regular bipartite Borel graphs without $k$-factors for all $k<d$ with $d$ even and $k$ odd.  Moreover, the graphs just mentioned are hyperfinite and do not admit Borel $k$-factors even after discarding a null or meager set.

There are at least two directions one can take in trying to extend this collection of counterexamples. One is to ask whether positive results can be obtained with additional structural assumptions on the graphs. Another is to ask about other values of $k$ and $d$. 

There are several recent positive results in the first direction. The connected components of the counterexamples in \cite{conley2013measurable} are all two-ended. In the Baire category and measurable settings with hyperfinteness, this appears to be necessary to some extent: Kechris and Marks \cite{kechris.marks} proved that every acyclic $d$-regular graph admits a Baire measurable $k$ factor for all $k<d$ for $d>2$. It follows from work of Conley, Marks, and Unger that the same is true for measurable $k$-factors of hyperfinite graphs \cite[Lemma 4.2]{conley_marks_unger}. Kastner and Lyons \cite{kastner2023baire} obtained a positive result for $k=1$ in the Baire measurable setting assuming only that their graphs had non-amenable components. Notably, their graphs are not necessarily bipartite. 

For graphs whose components are all one ended, Bowen, Kun, and Sabok \cite[Corollary 3.7]{bowen.kun.sabok} established the existence of measurable $k$-factors in regular bipartite pmp hyperfinite graphs.  It is worth remarking that in the pmp context, hyperfinite graphs have either one or two ends on almost every component \cite[Theorem 22.3]{kechris.miller}.  If the measure is not assumed invariant, as we consider below, hyperfiniteness no longer places constraints on the structure of ends, and the arguments of \cite{bowen.kun.sabok} do not directly apply.

In the one-ended context, Bowen, Poulin, and Zomback \cite{bowen2022one} established the analogous result for Baire measurable $k$-factors in regular bipartite Borel graphs on Polish spaces.

This paper complements these results by proving a positive result in the second direction. 
Indeed, our main result is that, in the Baire measurable setting and the measurable setting with hyperfiniteness, the values of $k$ and $d$ ruled out in the examples from \cite{conley2013measurable} are the only ones which cause problems.

\begin{theorem}\label{th:main}
    Let $0 \leq k \leq d \in \omega$ with $k$ even or $d$ odd, and let $G$ be a $d$-regular bipartite Borel graph on a standard Borel space $X$.
    \begin{itemize}
        \item If $\tau$ is a compatible Polish topology on $X$, then $G$ admits a Borel $k$-factor off a Borel $G$-invariant meager set.
        \item If $\mu$ is a Borel probability measure on $X$ for which $G$ is $\mu$-hyperfinite, then $G$ admits a Borel $k$-factor off a Borel $G$-invariant $\mu$-null set
    \end{itemize}
\end{theorem}

The case $k=1$ in the measurable setting was previously shown in \cite{bowen.kun.sabok} for $\mu$-preserving graphs. 


Kun \cite{gabor.new} has constructed for all $d$ a $d$-regular acyclic pmp Borel bipartite Borel graph with no measurable $k$-factor for any $0 < k < d$, showing that the hyperfiniteness assumption in the theorem is necessary.
 
The main novelty of the present work consists of applying to the bipartite setting work of \cite{conley.miller.pm} in the acyclic setting.  More precisely, \cite{conley.miller.pm} established that matchings in acyclic graphs exist when there are no rays of degree two on every other vertex.  We also employ a ``splitting trick'' (essentially the content of Subection \ref{sec:2match}) that reduces the problem of finding $k$-factors to that of finding well behaved fractional $1$-factors.

An additional source of interest in Theorem \ref{th:main} is the connection between Baire measurable combinatorics and computable combinatorics of so-called \textbf{highly computable} graphs. I.e, graphs on $\omega$ whose edge relation is a computable subset of $\omega^2$ and whose degree function $\omega \to \omega$ is computable. Qian and Weilacher \cite{qian2022descriptive} investigated this connection, and found that many ``local'' combinatorial problems behaved identically in the two settings.

Theorem \ref{th:main} gives the first examples of locally checkable coloring problems which are known to have a Baire measurable solution on Borel graphs but not computable solutions on highly computable graphs: Indeed, Manaster and Rosenstein \cite{manaster_rosenstein_regular} constructed for all $d > 1$ (crucially, including for $d$ odd) a $d$-regular computably bipartite computable graph with no computable perfect matching. Note that regularity implies that these graphs are in fact highly computable. It is easy to modify their proof to obtain such graphs which admit no computable $k$-factor for any $0 < k < d$. 



For non-bipartite graphs, odd-regularity does not guarantee the existence of a perfect matching, or more generally a $k$-factor for odd $k$. However, Petersen's 2-factor theorem \cite{petersen1891theorie} guarantees the existence of a $k$-factor in any $d$-regular graph with $k$ and $d$ even. Of course, the full theorem follows if one only considers the case $k = 2$. It is natural to ask for descriptive versions of this theorem, and we leave this as a question. 

\begin{question}
    Let $d \in \omega$ be even. Does every $d$-regular Borel graph on a Polish space admit a Borel 2-factor off an invariant meager set?
    
    What about in the measurable setting for hyperfinite graphs?
\end{question}

We mention the following partial result, which lends some support towards a positive answer to this question.

\begin{cor}
    Let $G$ be a $2k$-regular Borel graph which admits a Borel balanced orientation.  
    If $k$ is odd then $G$ admits a Borel 2-factor off an invariant meager set, 
    and if $k$ is even then $G$ admits a Borel 4-factor off an invariant meager set. 
    If $G$ is hyperfinite then the same is true in the measure setting.
\end{cor}

In particular, the Baire measurable part of the above Corollary applies to all one-ended Borel graphs and to all non-amenable Borel graphs by the results in \cite{bowen2022one} and \cite{kastner2023baire} respectively, and the measurable part to all one-ended hyperfinite measure preserving graphs by the results in \cite{bowen.kun.sabok}.  

\begin{proof}
    Given the orientation of $G,$ construct an auxiliary graph $G'$ with vertex set $\{v_1,v_2: v\in V(G)\}$ and edge set $\{\{u_1,v_2\}: \textnormal{ there is a directed } G-\textnormal{edge from } u \textnormal{ to } v\}. $  Since the orientation is balanced, $G'$ is a $k$-regular bipartite Borel graph, and so Theorem \ref{th:main} gives the desired result.
\end{proof}

{


\section{The proof of Theorem \ref{th:main}}

Observe that if Theorem \ref{th:main} is proved in the cases $k \in \{1,2\}$, the full theorem follows by induction, since if $G$ is a $d$-regular bipartite Borel graph and $H$ is a Borel $k$-factor, $G \setminus H$ is a $(d-k)$-regular bipartite Borel graph.  In Subsection \ref{sec:match} we will address the $k=1$ case, and in Subsection \ref{sec:2match} we will address the $k=2$ case.

Before beginning, we collect some preliminary notions that will be needed in both cases.


\begin{definition}

    A fractional $k$-matching is a symmetric function $f:G \to [0,1]$ such that for each vertex $x$,
    $$ \sum_{yGx} f(x,y) = k.$$
\end{definition}

For example, a fractional $1$-matching is what is usually called a fractional perfect matching. $k$-factors can be naturally identified with $\{0,1\}$-valued fractional $k$-matchings.

In our algorithms, our goal will be to ``round" fractional matchings until they become integral on certain edges.  In order to ensure convergence we will want the algorithm to stabilize on integral edges, and so to keep track of this we will use the following notion, introduced in \cite{bowen.kun.sabok}.

\begin{definition}
    Let $f$ be a fractional perfect matching of a graph $G$. The \textit{support} of $f$, denoted $\supp(f)$, is the subgraph of $G$ consisting of edges on which $f$ is not 0 or 1. 
\end{definition}

Note that $f$'s restriction to $\supp(f)$ is a fractional perfect matching of $\supp(f)$, 
and that
$\supp(f)$ is Borel if $G$ and $f$ are.  Further, $\supp(f)$ cannot have any degree one vertices, as the values of $f$ could not sum to $1$ in this situation.

\subsection{Perfect Matchings}\label{sec:match}


 In this subsection we handle the case $k = 1$ of Theorem \ref{th:main}.

Actually, we will start with a general statement about ``rounding'' fractional perfect matchings which will be key to both sections.

\begin{lemma}\label{lem:main}
    Let $G$ be a bipartite Borel graph, say on $X$, a with a Borel fractional perfect matching $f:G \to \{0, \frac{1}{d}, \frac{2}{d},\ldots, 1\}$ for some $d \in \omega$. Let $L = \{0,1\}$ if $d$ is odd and $L = \{0,\frac{1}{2},1\}$ if $d$ is even. Then $G$ admits a Borel $L$-valued fractional perfect matching\ldots
    \begin{itemize}
        \item \ldots off a Borel $G$-invariant meager set for any compatible Polish topology for $X$.
        \item \ldots off a Borel $G$-invariant null set for any Borel probability measure $\mu$ on $X$ for which $G$ is hyperfinite.
    \end{itemize}
\end{lemma}

Note that this immediately implies the $k=1$ case of Theorem \ref{th:main}, since in that case we assume $d$ is odd, and we can take $f$ to be the constant $\frac{1}{d}$ function.

In the pmp setting with hyperfiniteness, Lemma \ref{lem:main} was already known to hold by Theorem 1.3 of \cite{bowen.kun.sabok}.  In that context, every acyclic leafless graph must be a bi-infinite line, so Lemma \ref{lem:main} in the special case of pmp graphs follows from what is Claim \ref{acyclic} of our proof without needing the second half of our argument.  Outside of the pmp setting, we must deal with the case when $\supp(f)$ is an infinitely ended tree.  This is handled by a path decomposition argument and an analysis of matchings in acyclic graphs due to Conley and Miller \cite{conley.miller.pm}.

\begin{proof}[Proof of Lemma \ref{lem:main}]
    We may assume $G$ is locally finite by ignoring edges on which $f = 0$. We start with the following reduction.

    \begin{claim}\label{acyclic}
        It suffices to prove the result when $\supp(f)$ is acyclic.
    \end{claim}

    \begin{proof}
    We will, after throwing away a meager/null Borel $G$-invariant set, replace $f$ with another $\{0,\frac{1}{d},\ldots,1\}$-valued Borel fractional perfect matching, call it $g$, such that $\supp(g)$ is acyclic. In the measurable setting, this is done in the proof of Theorem 1.3 in \cite{bowen.kun.sabok}. (That proof assumes measure preservation, but this step does not use this assumption.)

    The same idea gives a proof in the category setting: Let $\sigma = (\sigma_0,\sigma_1) \in (\omega \times 2)^\omega$. We will build a uniformly Borel sequence of $\{0,\frac{1}{d},\ldots,1\}$-valued
    fractional perfect matchings $f = f_0^\sigma, f_1^\sigma, \ldots$ with $\supp(f_{i+1}) \subseteq \supp(f_i)$ for each $i$. First, using a lemma of Kechris-Miller (see, e.g., \cite[Proposition 3]{conley.miller.toast}), we may fix a Borel $\omega$-coloring $c$ of the set of cycles of $G$ such that any cycles sharing a vertex get different colors. Also fix a Borel function $\phi$ which selects an edge from each cycle of $G$ using the Lusin-Novikov uniformization theorem.

    Given $f_i^\sigma$, let $S$ be the set of cycles $C$ of $G$ with $c(C) = \sigma_0(i)$ and $C \subseteq \supp(f_i)$. For any $C \in S$, since $f_i^\sigma$ takes only the values $\{\frac{1}{d},\ldots,\frac{d-1}{d}\}$ on $C$ and $C$ is even, we can alternate adding and subtracting $\frac{1}{d}$ from the edges of $C$ and still get a fractional perfect matching, and this does not increase the support. There are two ways of doing this: one where the value on $\phi(C)$ increases and one where it decreases. Let us do the former if $\sigma_1(i) = 0$ and the latter otherwise. Since the cycles in $S$ are pairwise disjoint, we can do this for all of them simultaneously. This will be our $f_{i+1}^\sigma$.

    Let $e$ be an edge of $G$. Call $\sigma$ \textit{good} for $e$ if the value $f_i^\sigma(e)$ eventually stabilizes and every cycle containing $e$ is eventually not contained in $\supp(f_i^\sigma)$. We claim the set of good $\sigma$ for a given $e$ is comeager: indeed, given a finite initial segment of $\sigma$, there is an extension of length at most $d$ which forces this to be the case: it corresponds to always choosing the color of some cycle in $\supp(f_i)$ containing $e$, if there is one, and always choosing the value at $e$ to decrease: this will either cause the value at $e$ to eventually be 0, at which point it will stabilize, or it will stop when there are no more such cycles. 

    Therefore, by the Kuratowski-Ulam theorem and local countability we can find a $\sigma$ and a Borel comeager $G$-invariant set on which the $f_i^\sigma$'s are good for each edge. We can then define $g$ to be their limit on this set, and $\supp(g)$ will be acyclic by the definition of good.
    \end{proof}

    Thus, assume that $\supp(f)$ is acyclic. 
    Recall that it has no degree 1 vertices. 
    Much of the work will now be done by the main results from \cite{conley.miller.pm}. To apply them, we need to say something about infinite rays through $\supp(f)$ which have degree 2 on every other vertex. Call such a ray a \textit{bad ray}

    \begin{claim}
        Every bad ray has degree 2 on all but finitely many indices. 
    \end{claim}

    \begin{proof}
        Let $x_0,x_1,x_2,\ldots$ be a bad ray, say where each $x_n$ for $n$ even has $\supp(f)$-degree 2. For each $n \in \omega$, let $w(n) = f(x_{2n},x_{2n+1})$. 
        Suppose $n$ is such that $x_{2n+1}$ has $\supp(f)$-degree 2, so that its only neighbors are $x_{2n}$ and $x_{2n+2}$. Then $f(x_{2n+1},x_{2n+2}) = 1 - w(n)$, so $w(n+1) = f(x_{2n+2},x_{2n+3}) = w(n)$. On the other hand, suppose $x_{2n+1}$ has degree $>2$, so that it has some other $\supp(f)$-neighbor, say $y$. Since $f(x_{2n+1},y) > 0$ by definition of support, $f(x_{2n+1},x_{2n+2}) < 1 - w(n)$, so $w(n+1) > w(n)$. Since $w$ takes values in $\{\frac{1}{d},\ldots,\frac{d-1}{d}\}$, this increase can happen only finitely many times.
    \end{proof}

    We will now describe how to get the promised $g$. Let $X' \subseteq X$ be the union of $\supp(f)$-components which are not bi-infinite paths, and consider the graph $H := \supp(f) \res X'$. By the claim, every tail equivalence class of bad rays in $X'$ is uniquely represented by a bad ray $(x_0)_{n \in \omega}$ such that $\deg(x_0) = 3$ and $\deg(x_n) = 2$ for all $n > 0$. Let $H'$ be the induced subgraph of $H$ obtained by deleting $x_n$ for $n > 1$ for each such representative. Let $Y$ be the set of $x_1$'s for such representatives. Then each $x \in Y$ has $H'$-degree 1, and all other vertices in $H'$ have the same degree they had in $H$. It follows that $H'$ has no bad rays. 

    It now follows from the main results in \cite{conley.miller.pm} that $H'$ admits a Borel matching, say $M'$, off a meager/null (the latter if it is hyperfinite) invariant set which covers all vertices not in $Y$. For example, we may attach a 3-regular tree of new vertices to each $x \in Y$ with $x$ as the root to obtain a Borel acyclic graph with all vertices having degree $\geq 2$ and no bad rays, then find a perfect matching of this graph, then consider the intersection of this matching and $H'$.  

    This extends to a Borel perfect matching, say $M$, of $H$ off a meager/null invariant set. For each representative $(x_0)$ as before, if $(x_0,x_1) \in M'$, we add $(x_n,x_{n+1})$ to $M$ for each even $n$. Else, for each odd $n$. On the components of $H$, we set $g$ to be the characteristic function of $M$.  
    

    The remaining components of $\supp(f)$ are bi-infinite paths. If $d$ is odd, $f$ must alternate between $<\frac{1}{2}$ and $>\frac{1}{2}$ on the edges in such components, so we can define $g$ on such components as $f$ rounded to the nearest integer. If $d$ is even, then we are allowed to set $g = \frac{1}{2}$ on all the edges in such components. 
\end{proof}

\subsection{2-factors}\label{sec:2match}

In this subsection we prove the $k=2$ case of Theorem \ref{th:main}. If $d$ is odd, then by the previous case, $G$ admits a Borel $(d-1)$-factor, say $H$, off a
meager/null set, and then a 2-factor of $H$ is a 2-factor of $G$. Thus we may assume $d$ is even. 

\begin{lemma}\label{lem:tech}
    Let $G$ be a bipartite Borel graph, say on $X$, with a Borel fractional $2$-matching $f: G \to \{0,\frac{1}{m},\ldots,1\}$ for some $m$ odd 
    with the property that for every vertex $x \in X$, there is a partition of the edges incident to $x$ into two sets such that the values of $f$ in each set sum to 1. Then $G$ admits a Borel 2-factor\ldots
    \begin{itemize}
        \item \ldots off a Borel $G$-invariant meager set for any compatible Polish topology for $X$.
        \item \ldots off a Borel $G$-invariant null set for any Borel probability measure $\mu$ on $X$ for which $G$ is hyperfinite.
    \end{itemize}
\end{lemma}

\begin{proof}
    For each $x \in X$, let $(P_x^0, P_x^1)$ be a partition of the edges incident to $x$ as in the lemma statement. By the Lusin-Novikov uniformization theorem, we can assign these partitions in a Borel fashion. We will define a Borel graph, say $G$, whose vertex set is $X \times 2$. (The obvious choices for measure or topology here will be fine.) For each edge $e = (x,y) \in G$, include a corresponding edge $e' = ((x,i),(y,j))$ in $G'$, where $i,j$ are such that $e \in P_x^i,P_y^j$. $G'$ is bipartite since the first projection $(X \times 2) \to X$ is a homomorphism from $G'$ to $G$. By hypothesis, the function $e' \mapsto f(e)$ is a Borel fractional perfect matching of $G'$. By Lemma \ref{lem:main}, since $m$ is odd, $G'$ has a Borel perfect matching off a meager/null set, call it $H' \subseteq G'$. But now $\{e \in G \mid e' \in H'\}$ is a Borel 2-factor of $G$ off a meager/null set. 
\end{proof}

We now show how to build such an $f$, possibly after throwing away a meager/null set, for any $d$-regular graph:

\begin{proof}[Proof of Theorem \ref{th:main}]
    
    

    Recall that $d$ is even. Apply Lemma \ref{lem:main} to get a fractional perfect matching $g : G \to \{0,\frac{1}{2},1\}$ off a meager/null set. We now define a new function $f$ which has value $\frac{1}{d-1}$ on $g^{-1}(0)$, $\frac{(d/2)}{d-1}$ on $g^{-1}(1/2)$, and $1$ on $g^{-1}(1)$. This is a fractional 2-matching because each vertex either has one incident edge with $g = 1$ and $d-1$ with $g = 0$, or two with $g = 1/2$ and $d-2$ with $g = 0$. Also, to witness the condition of Lemma \ref{lem:tech}, for the first type of vertex we can take a partition where one of the sets consists of the unique incident edge in $g^{-1}(0)$, and for the second type of vertex we can have each set in the partition consist of one incident edge from $g^{-1}(1/2)$ and $\frac{d-2}{2}$ from $g^{-1}(0)$.    
\end{proof}



\bibliographystyle{amsalpha} 
\bibliography{main}

\providecommand{\bysame}{\leavevmode\hbox to3em{\hrulefill}\thinspace}
\providecommand{\MR}{\relax\ifhmode\unskip\space\fi MR }
\providecommand{\MRhref}[2]{%
  \href{http://www.ams.org/mathscinet-getitem?mr=#1}{#2}
}
\providecommand{\href}[2]{#2}
\begin{thebibliography}{CMU20}

\bibitem[BKS21]{bowen.kun.sabok}
Matthew Bowen, G{\'a}bor Kun, and Marcin Sabok, \emph{Perfect matchings in
  hyperfinite graphings}, arXiv preprint arXiv:2106.01988 (2021).

\bibitem[BPZ22]{bowen2022one}
Matthew Bowen, Antoine Poulin, and Jenna Zomback, \emph{One-ended spanning
  trees and definable combinatorics}, arXiv preprint arXiv:2210.14300 (2022).

\bibitem[CK13]{conley2013measurable}
Clinton~T Conley and Alexander~S Kechris, \emph{Measurable chromatic and
  independence numbers for ergodic graphs and group actions}, Groups, Geometry,
  and Dynamics \textbf{7} (2013), no.~1, 127--180.

\bibitem[CM16]{conley.miller.toast}
Clinton~T. Conley and Benjamin~D. Miller, \emph{A bound on measurable chromatic
  numbers of locally finite {B}orel graphs}, Math. Res. Lett. \textbf{23}
  (2016), no.~6, 1633--1644.

\bibitem[CM17]{conley.miller.pm}
\bysame, \emph{Measurable perfect matchings for acyclic locally countable
  {B}orel graphs}, J. Symb. Log. \textbf{82} (2017), no.~1, 258--271.

\bibitem[CMU20]{conley_marks_unger}
Clinton~T Conley, Andrew~S Marks, and Spencer~T Unger, \emph{Measurable
  realizations of abstract systems of congruences}, Forum of Mathematics,
  Sigma, vol.~8, Cambridge University Press, 2020, p.~e10.

\bibitem[KL23]{kastner2023baire}
Alexander Kastner and Clark Lyons, \emph{Baire measurable matchings in
  non-amenable graphs}, arXiv preprint arXiv:2310.20047 (2023).

\bibitem[KM04]{kechris.miller}
Alexander~S. Kechris and Benjamin~D. Miller, \emph{Topics in orbit
  equivalence}, Lecture Notes in Mathematics, vol. 1852, Springer-Verlag,
  Berlin, 2004.

\bibitem[KM16]{kechris.marks}
Alexander~S. Kechris and Andrew~S. Marks, \emph{Descriptive graph
  combinatorics}, 2016, preprint.

\bibitem[Kun21]{gabor.new}
G{\'a}bor Kun, \emph{The measurable hall theorem fails for treeings}, arXiv
  preprint arXiv:2106.02013 (2021).

\bibitem[Lac88]{lacz3}
M.~Laczkovich, \emph{Closed sets without measurable matching}, Proc. Amer.
  Math. Soc. \textbf{103} (1988), no.~3, 894--896.

\bibitem[MR73]{manaster_rosenstein_regular}
Alfred~B. Manaster and Joseph~G. Rosenstein, \emph{Effective matchmaking and
  k-chromatic graphs}, Proceedings of the American Mathematical Society
  \textbf{39} (1973), no.~2, 371--378.

\bibitem[Pet91]{petersen1891theorie}
Julius Petersen, \emph{Die theorie der regul{\"a}ren graphs}, Acta Mathematica
  \textbf{15} (1891), 193--220.

\bibitem[QW22]{qian2022descriptive}
Long Qian and Felix Weilacher, \emph{Descriptive combinatorics, computable
  combinatorics, and asi algorithms}, arXiv preprint arXiv:2206.08426 (2022).

\bibitem[Tut54]{tutte1954short}
William~Thomas Tutte, \emph{A short proof of the factor theorem for finite
  graphs}, Canadian Journal of mathematics \textbf{6} (1954), 347--352.

\end{thebibliography}

\end{document}